\documentclass{endm}
\usepackage{endmmacro}
\usepackage{graphicx,float}
\usepackage{amsmath,amsfonts,amssymb}
\usepackage{algorithm}
\usepackage{algorithm}
\usepackage{float}
\usepackage{algpseudocode}

\newcommand{\C}{{\mathcal C}}
\newcommand{\T}{{\mathcal T}}
\newcommand{\Real}{{\mathbb R}}
\newcommand{\one}{{\mathbf 1}}
\newcommand{\F}{\mathcal{F}}

\begin{document}

% DO NOT REMOVE: Creates space for Elsevier logo, ScienceDirect logo
% and ENDM logo
\begin{verbatim}\end{verbatim}\vspace{2.5cm}

\begin{frontmatter}
\title{Packing functions  and  graphs with perfect closed neighbourhood matrices\thanksref{ALL}}

\author{M. Escalante$^{a,b}$ \; \; E. Hinrichsen$^{a}$ \; \; V. Leoni$^{a,b}$}

\address[a]{ 
Depto. de Matem\'atica, FCEIA. Universidad Nacional de Rosario, Argentina}

\address[b]{Consejo Nacional de Investigaciones Cient\'ificas y T\'ecnicas, Argentina}

\thanks[ALL]{Partially supported by grants PICT ANPCyT 0410, PID UNR 1ING631 and PID UNR ING542}

\thanks[coemail]{Emails:
{\{mariana,ericah,valeoni\}@fceia.unr.edu.ar} }

\begin{abstract} 
  In this work we consider a straightforward linear programming formulation of the recently introduced $\{k\}$-packing function problem in graphs, for each fixed value of the positive integer number $k$. We analyse a special relation between the  case $ k = 1$ and $ k \geq 2$ and give a sufficient condition  for optimality ---the perfection--- of the closed neighbourhood matrix $N[G]$ of the input graph $G$.
We begin a structural study of graphs satisfying this condition. In particular, we look for a characterization of graphs that have perfect  closed neighbourhood matrices which involves the property of being a clique-node matrix of a perfect graph. We present a necessary and sufficient condition for a graph to have a clique-node closed neighbourhood matrix. Finally, we study the perfection of the graph of maximal cliques associated to $N[G]$. 
\end{abstract}
\end{frontmatter}

\section{Preliminaries and notation}\label{def}

\subsection{Graphs and binary matrices}

Throughout this work we consider finite simple connected graphs $G=(V,E)$, where $V$ is its node set and $E$ is its edge set. When we need to emphasize the relationship between the graph $G$ and its node and edge sets, we write $V(G)$ and $E(G)$.

The \emph{complementary graph} of $G$, denoted by $\overline{G}$, is such that $V(\overline{G})=V(G)$ and $E(\overline{G})=\{uv:  u,v\in V(G), uv \notin E(G)\}$.  

Given $G=(V,E)$ and $V'\subseteq  V$, we denote by $G[V']$ the subgraph of $G$ induced by the nodes in $V'$, i.e., the graph having node set $V'$ and edge set  $\{uv: uv\in E, \{u,v\}\subseteq  V'\}$. Then $G'$ is a \emph{ node induced subgraph} of $G$, $G'\subset G$, if $G'=G[V']$ for  some $V'\subseteq  V$. 

Given the graph $G=(V,E)$, the \emph{neighbourhood} of $v\in V$ is  $N(v)=\{u\in V: uv \in E \}$ and its \emph{closed neighbourhood} is $N[v] = N(v) \cup \{v\}$.  When we need to emphasize the graph $G$, we write $N_G(v)$ (respectively, $N_G[v]$) instead.
A vertex $v \in V$ is \emph{universal} if $N[v]=V$. 

For any positive integer number $n$, $K_n$ denotes the graphs with $n$ distinct nodes corresponding to a  \emph{complete graph}, i.e. a graph where for all distinct nodes $u,v \in V(K_n)$, $uv\in E(K_n)$.

A \emph{cycle}, denoted by $C_n$, is a graph with $n$ distinct nodes, say $1,\dots,n$, and edges $i(i+1)$ for $i=1,\dots,n-1$ together with the edge $1n$.  We say that $i$ and $i+1$, and $1$ and $n$  are consecutive nodes in the cycle. The \emph{length} of the cycle is the number  of its edges. An \emph{even} (\emph{odd}) cycle is a cycle of even (odd) length.

A \emph{chord} in a cycle is an edge connecting two non consecutive nodes in the cycle.

A \emph{hole} (resp. \emph{anti-hole}) is an induced cycle (the complementary graph of a cycle) in the graph (i.e. a hole has no chords).

A graph $G$ is \emph{chordal} if every cycle in $G$ of length at least four  have a chord. $G$ is \emph{strongly chordal} if it is chordal and also every cycle of even length at least 6 in $G$ has an odd chord, i.e., an edge that connects two nodes that are an odd distance apart from each other in the cycle.

A \emph{clique} in $G$ is a subset of pairwise adjacent nodes in $G$. 

Although it is not the original definition of perfection in graphs, we say that a  graph is \emph{perfect} if it has neither an odd hole nor an odd anti-hole as an induced subgraph, due to the Strong Perfect Graph Theorem  \cite{Chud05}.

In this work we make use of strong ties between all the above mentioned graph definitions from the polyhedral point of view, and then we need to present some $0,1$-matrices associated with graphs. 
   
Given $G=(V,E)$, its \emph{closed neighbourhood matrix} $N[G]=(a_{ij})$ is a $0,1$-matrix with rows and columns indexed in $V$ and such that  $a_{ij}=1$ when $i=j$ or $ij \in E(G)$ and $a_{ij}=0$ otherwise. Throughout this work we can use a different permutation of the node set $V$ in rows and columns, but in this case it is mentioned explicitly. 

Matrix $J$ denotes  the square matrix  whose entries are all 1's, and $I$ the identity matrix, both of appropriate sizes.

A $0,1$-matrix $M$ having $m$ rows and $n$ (non zero) columns  is \emph{perfect} 
if $P(M)=\{x\in [0,1]^n: Mx\leq \one\}$ is an integer polyhedron, i.e. all extreme points of $P(M)$  are $0,1$-vectors, where $\one$ denotes the vector of all 1's of appropriate size.

The \emph{clique-node} matrix of a graph $G$, denoted by $\mathcal{C}(G)$, is the $0,1$-matrix whose columns are indexed by the nodes of $G$ and whose rows are the incidence vectors of maximal cliques in $G$. 
Chv\'atal \cite{Chv75} proved that a $0,1$-matrix with no zero columns and no dominating rows is a perfect matrix if and only if it is the clique-node matrix of a perfect graph.% The recognition problem for perfect matrices is polynomial \cite{Chud05}. 

We say that a matrix is an \emph{extended clique-node} matrix of $G$ if it is a square $0,1$-matrix that contains $\mathcal{C}(G)$ as a row submatrix and all its remaining rows (if any) are  incidence vectors of cliques in $G$.
Note that the presence of dominating rows do not affect the integrality of the polyhedrom $P(M)$ and thus, Chv\'atal's result  mentioned  above can be restated as the following 

\begin{theorem}\cite{Chv75}\label{ch} A square $0,1$-matrix $M$ with no zero columns is a perfect matrix if and only if it is an extended clique-node matrix of a perfect graph.
\end{theorem}

The following characterization, derived from a similar result for clique-node matrices exposed in \cite{cornuejols}, is crucial for our analysis, since it yields a polynomial time algorithm to check if a $0,1$-matrix is an extended clique-node matrix. 

\begin{theorem}\cite{cornuejols}\label{cliquenodo}
Let $M$ be a $0,1$-matrix. The following statements are equivalent:

\begin{enumerate}
\item $M$ is an extended clique-node matrix,
\item if 
 
\begin{equation}\label{J-I}
\left(
\begin{array}{ccccccccc}
0&&1&&1&& 1& \dots &1\\
1&&0&&1&&1& \dots &1\\
1&&1&&0&&1& \dots &1
\end{array}
\right)
\end{equation}
is a submatrix of $M$, say, with columns $j_1,\dots,j_p$, where $p\geq 3$, then $M$ contains a row $i$ such that $m_{ij_k}=1$ for $k=1,\dots,p$.
\item If $J-I$ is a $p\times p$ submatrix of $M$, where $p\geq 3$, then $M$ contains a row $i$ such that $m_{ij}=1$ for every column $j$ of $J-I$.
\end{enumerate}
\end{theorem}

Given a $0,1$-matrix $M$, the \emph{clique graph} of $M$, $G_Q(M)$, is the  graph having one node for each column of $M$ and $ij$ being an edge  when there is a row having value 1 in the $i$-th and $j$-th columns of $M$.

Under this definition, the two last conditions in Theorem \ref{cliquenodo} say that the columns in $M$ indexed in $\{j_1,\dots,j_p\}$ form  a clique in the graph $G_Q(M)$. A necessary and sufficient condition for $M$ to be an extended  clique-node matrix is that the incidence vector of every of these  cliques are also rows of $M$. 

According to Chv\'atal's result, the matrix $M$ is perfect if and only if $M$ is the clique-node matrix of $G_Q(M)$ and $G_Q(M)$ is a perfect graph.

When $M=N[G]$ for some graph $G$, we simply write $G_Q$ instead of $G_Q(M)$, and refer to it as the \emph{clique graph associated with $G$}.

\subsection{Limited packing functions}
 
For a graph $G=(V,E)$ and a positive integer $k$, a function $f : V  \mapsto \mathbb{Z}_+$ is a \emph{$\{k\}$-packing function} of $G$ \cite{LNCS14HL}, if for each node $v\in V$, 
\[
f(N[v])= \sum_{w\in N[v]} f(w) \leq k.
\]
The maximum possible value of $f(V)= \sum_{v\in V} f(v)$ over all $\{k\}$-packing functions $f$ of $G$ is denoted as $L_{\{k\}}(G)$. 

These concepts generalize the notion of $k$-limited packing introduced  by Gallant et al. \cite{GGHR1}. A \emph{$k$-limited packing} in $G$ is a function $f:V\rightarrow\{0,1\}$ such that for each node $v\in V$, 
\[f(N[v])\leq k.\]
Again, the maximum possible value of $f(v)$ over all $k$-limited packings $f$ in $G$ is denoted by $L_k(G)$.  

From their definitions, it holds $L_1(G)=L_{\{1\}}(G)$ and $L_k(G) \leq L_{\{k\}}(G)$ for every graph $G$. Nevertheless, this last inequality may or may not be  strict. In Figure \ref{kpacking} there is  a graph $G$, an optimal $\{3\}$-packing function and a $3$-limited packing in $G$ whose optimal values do not satisfy the equality. In the following two figures we have represented the value of the function next to each of the nodes.

\begin{figure}[ht]
\centering
\includegraphics[scale=1.1]{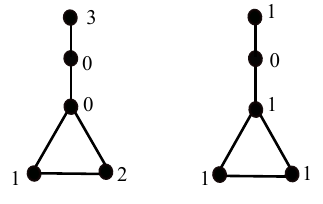}
\caption{A  graph $G$ with  $L_{\{3\}}(G)=6$ and $L_3(G)=4$.}\label{kpacking}
\end{figure}

The optimization problem associated with this concept is the following:
\bigskip

\noindent
{\textbf{$\{k\}$-PACKING FUNCTION}} {\scshape ($\{k\}$PF}, for fixed positive  integer $k$):
 Given a graph $G$, find  $L_{\{k\}}(G)$.

\bigskip

By using simple combinatorial arguments, in \cite{DHLITOR2017} it is proved  that $L_{\{k\}}(G)\geq k L_1(G)$, for every integer number $k$ and every graph $G$. Nevertheless, note that this inequality may be strict as the following example shows.

\begin{example}\label{graphS3}
Let $G$ be the graph 
 $S_3$ depicted in Figure \ref{otro}.  
It is clear that $L_1(S_3)=1$ but $L_{\{3\}}(S_3)=4$, as shown in the figure.  
\end{example}

\begin{figure}[ht]
\centering
\includegraphics[scale=0.8]{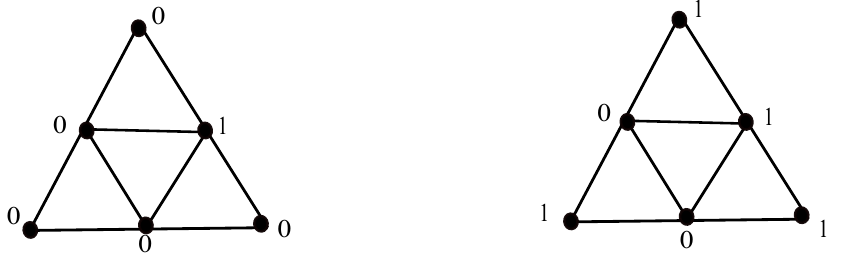} 
\caption{Optimal $\{1\}$- and $\{3\}$-packing functions of $S_3$, respectively.}\label{otro}
\end{figure}

Besides, it is known that $\{k\}$PF is linear time solvable  in strongly chordal graphs \cite{LNCS14HL} and it is NP-complete for chordal graphs \cite{DHLITOR2017}. 

We present a natural Integer Linear Programming formulation of $\{k\}$PF for fixed $k$. We give a sufficient condition for optimality, the perfection of the closed neighbourhood matrix of the input graph. 

A different sufficient condition involving dominating sets was already given in  \cite{DHLITOR2017}. 
The step forward with this contribution is that, knowing that checking matrix perfection can be performed in polynomial time for any input graph (see Theorem \ref{cliquenodo}), we are allowed to find new polynomial instances of problem $\{k\}$PF.   In the last section, we begin a structural study of graphs satisfying the new sufficient condition and present a characterization of graphs that have closed neighbourhood matrices that are also clique-node matrices and necessary conditions for the associated clique graphs to be perfect.

\section{A sufficient condition for optimality of the $\{k\}$-Packing Function Problem}

We now  consider an integer programming formulation for $\{k\}$PF. 

Given $G=(V,E)$ with $|V|=n$, let $x_i=f(i)$ for $i=1, \dots, n$. 
For $u,v \in \Real^n$,  $u \cdot v=\sum_{i=1}^n  u_i v_i$, and the associated matrix-vector product.

For fixed integer positive number $k$, $\{k\}$PF can be formulated as follows:

\[ L_{\{k\}}(G)= \max \{ \one \cdot x:  N[G]\cdot x \leq k \one,\quad x\in \{0, 1, \dots, k\}^n\}.
\]

Let $L^R_1(G)$ be the optimal value of the linear relaxation of this problem for $k=1$, i.e. allowing $x$  to take any value in $[0,1]^n$. Then, we have:

\begin{theorem}\label{suf2}
For each fixed positive integer $k$ and graph $G$, if $N[G]$ is perfect then $L_{\{k\}}(G)$ can be found in polynomial time and $L_{\{k\}}(G)=k  L_1(G)$. 
\end{theorem}

\begin{proof}
It is clear that, if $x$ is a solution of the linear relaxation of the problem above for $k=1$, then $k N[G]\cdot x \leq k  \one$ and thus $kx$ is a solution of the linear relaxation of $\{k\}$PF, and conversely. Consequently, 
 \[L_{\{k\}}(G)  \leq k  \max \{\one \cdot y : N[G]\cdot y \leq \one, \; y\in [0,1]^n\}= k L^R_1(G).\] 

Notice that if $N[G]$ is perfect, then $L_1(G) =L_1^R(G)$  and, therefore, $L_{\{k\}}(G)\leq k L_1(G)$.  
Recall that $L_{\{k\}}(G)\geq k L_1(G)$, for each $k$ and every graph $G$ \cite{DHLITOR2017}.

In all, we have the desired equality.
\end{proof}

Notice that the cycle $C_4$ is an example that shows that the sufficient condition in Theorem \ref{suf2} is not necessary. In fact, 
  $L_{\{k\}}(C_4)=k$ for every $k$ that is not a mutiple of three and $L_1(C_4)=1, $ but $N[C_4]$ is not perfect since it is not a clique-node matrix.

In the next section we focus on the family of graphs for which their closed  neighbourhood matrices are perfect.

\section{Perfect closed neighbourhood matrices}\label{carac}

In this section we begin a structural study of graphs satisfying the sufficient condition in Theorem \ref{suf2}.

Let  $\F$ be the family of graphs whose closed neighbourhood matrices are perfect. 
There are some trivial graphs in the family $\F$, for instance complete graphs. 
Let us now introduce some conditions for a graph $G$ to belong to the family $\F$. 

Recall that, given $G$, we consider the clique graph associated with it, denoted by $G_Q$. Then, a graph $G$ belongs to $\F$ if and only if $N[G]$ is an extended clique-node matrix and $G_Q$ is a perfect graph.

Observe that the characterization of extended clique-node matrices shows that if $G=(V,E) \in \F$, for every maximal clique in $G_Q$ there must be  a node in $V$ adjacent to all the nodes in this clique.

Therefore, 

\begin{lemma}
If a graph $G=(V,E)$ has a universal node then it belongs to $\F$. 
\end{lemma}

\begin{proof}
Let $N[G]=(a_{ij})$, for $i,j\in V$ and $v$ a universal node in $G$. Then, $a_{vj}=1$ for all $j\in V$. 

Theorem \ref{cliquenodo} implies that  $N[G]$ is a clique-node matrix and  the graph $G_Q$ is a complete graph, thus perfect. 
\end{proof}

 A \emph{wheel} graph $W_n$ on $n$ nodes  is a cycle on $n-1$ nodes together with a universal node outside the cycle. 
From the previous result, wheel graphs are also in $\F$. 
However, a wheel of at least five nodes is not even chordal. Therefore, there are some trivial non chordal graphs in $\F$.

A particular subset of chordal graphs is the set of \emph{strongly chordal} graphs. A  well-known characterization of them is given by the total balancedness of their closed neighbourhood matrices \cite{farber}. A $0,1$-matrix is \emph{totally balanced} if it does not contain as a submatrix the (node-edge) incidence matrix of a cycle of length at least three.
Due to Farber  \cite{farber}, a graph is totally balanced if and only if its closed neighbourhood matrix is totally balanced.  
The fact that strongly chordal graphs are in $\F$ follows from the fact that totally balanced matrices are perfect \cite{fulkerson}. 

Clearly, the family of chordal graphs and $\F$ do not coincide but intersect each other in, at least, strongly chordal graphs. 

The family of web graphs generalize other families of graphs having strong symmetry, like complete graphs, cycles and the complementary graphs of cycles. Therefore it is important to determine which of them belong to $\F$.
 
Given positive integer numbers $n$ and $k$, the \emph{web} graph $W_n^{k}$ has $n$ nodes, say  $1, \dots, n$, and $ij$ is an edge if $i$ and $j$ differ by at  most $k$ (mod $n$) and $ i \neq j$. The complementary graph of a web is an \emph{antiweb}. Antiwebs have also a circular symmetry of their maximum cliques and stables sets. If $n\leq 2k+1$ the web $W_n^k$ is the complete graph of $n$ nodes. Clearly, for $k\geq 2$ the graph  $W_{2k+1}^1$ is an odd cycle and $W_{2k+1}^{k-1}$ is the complementary graph of an odd cycle.

A \emph{circulant matrix} $C_{n}^k$, for $ k\in\{1,\dots, n-1\}$ is the $0,1$-matrix having columns indexed in $\{1,\dots,n\}$ and rows given by all the incidence vectors of $\{i+1,\dots,i+k\}$ for $i\in\{1,\dots,n\}$, addition mod $n$. 

\begin{remark}\label{circulante}
There is a strong relationship between circulant matrices and web graphs. In fact, $C_n^{k+1}$ is a clique-node matrix  if and only if $n\geq 3k+1$. Moreover, in this case, $C_n^{k+1}=\C(W_n^k)$. 
This follows from the fact that, when $n\geq 3k+1$ the maximal cliques in $W_n^k$ are defined by $k+1$ consecutive nodes in the web and when $n\leq 3k$ we have some more maximal cliques, which are defined by appropriate non consecutive nodes, as in the case of $W_7^3$. 
\end{remark}

Also, perfect  web graphs are completely identified. 

\begin{remark}\label{perfect_webs}
The authors in \cite{paper_su} completely characterized the disjunctive rank of the clique relaxation of the stable set polytope, i.e., the minimum number of nodes that must be deleted from a web in order to become a perfect graph. As a consequence, they proved that a web graph  $W_n^k$ is perfect if and only if it is a complete graph or $n=2(k+1)$.
\end{remark}

Using the results in these two remarks we can  prove the following

\begin{proposition}\label{webs}
Complete graphs are the only web graphs in $\F$.
\end{proposition}
\begin{proof}
We have already mentioned that complete graphs belong to $\F$.

Let $W_n^k$ be a web graph with $n>2k+1$ and $k\geq 1$.
It is clear that $N[W_n^k]=C_n^{2k+1}$. After Remark \ref{circulante}, $C_n^{2k+1}=\C(W_n^{2k})$ if and only if $n\geq 6k+1$.
In order $W_n^{k}$ to belong to $\F$, $N[W_n^{k}]$ has to be the clique-node matrix of a perfect graph. Remark \ref{perfect_webs} implies that $n=4k+2$, which contradicts the fact that $n\geq 6k+1$.
\end{proof}

Due to their importance on the sequel, we present some small examples of graphs whose  closed neighbourhood matrices are not extended clique-node matrices.

\begin{example}\label{malos1}
The closed neighbourhood matrices of the chordless cycles $C_n=W^1_n$, for $n\in \{4,5,6\}$ are not clique-node matrices.
This holds since $N[W^1_n]=C_n^3$ for $n\in \{4,5,6\}$, and they are not clique-node matrices after Remark \ref{circulante}. 
\end{example}

\begin{example}\label{malos11}
Consider again graph $S_3$ depicted in Figure \ref{minimal}. From Theorem \ref{suf2} and taking into account Example \ref{graphS3}, it is clear that $N[G]$ is not a perfect matrix. 

Also, this can be seen regarding at the neighbourhood matrix, since there are three  nodes of degree 4 in $G$ defining a structure in $N[G]$ as in \eqref{J-I} for $p=6$ but $N[S_3]$ has no row of all ones and thus it is not an extended clique-node matrix. 
\end{example}

\begin{figure}[h]
\centering
\includegraphics[scale=0.8]{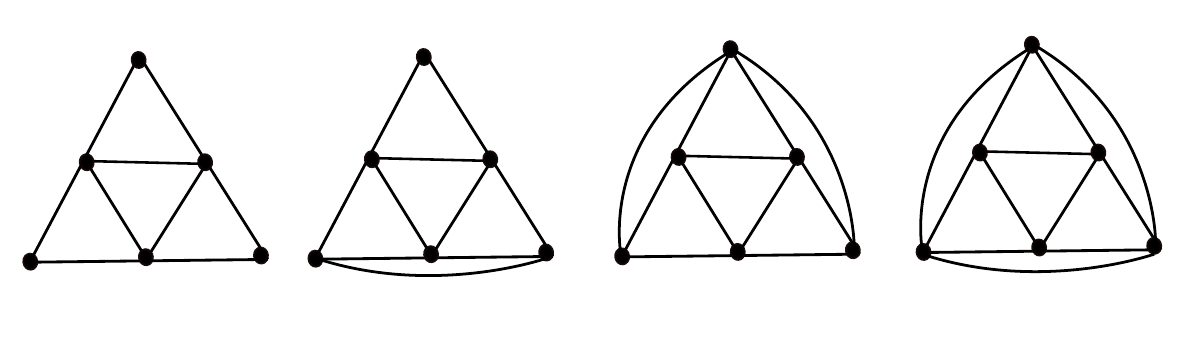}
\caption{Graphs $S_3$ (known also as 3-sun or Haj\'os graph), 1-pyramid, 2-pyramid and 3-pyramid.}\label{minimal}
\end{figure}

If $\T=\{C_4, C_5, C_6, S_3\}$,  we next show that a necessary and sufficient condition for $N[G]$ to be an extended clique-node matrix is that all the induced subgraphs of $G$ in $\T$ (if any) have a universal node in $G$.

\begin{theorem}\label{main}
Let $G=(V,E)$ be a graph.  $N[G]$ is an extended clique-node matrix if and only if for every $G'\subset G$ and $G'\in \T$, there exists $v\in V\setminus V(G')$ such that 
$N(v)\supset V(G')$.
\end{theorem}

\begin{proof}

Suppose $N[G]$ is an extended clique-node matrix. 
From examples \ref{malos1} and \ref{malos11}, it is clear that $G$ does not belong to $\T$. 

Let $G'\subset G$ with  $G'\in \T$ and consider its neighbourhood matrix  $N[G']$:

\begin{itemize}
\item If $G'=C_4$ then $N[G']=J-I$.  Theorem \ref{cliquenodo} with $p=4$ implies that  there must be a row of $N[G]$ corresponding to nodes not in $V(G')$ with 1's in the columns associated with  the nodes in $G'$, say $j_1, j_2, j_3, j_4$.  
This means that there is a node $v\in V\setminus V(G')$  adjacent to every node in $V(G')$ and the result holds.

\item If $G'=C_5$ then $N[G']=C^3_5$  which is not an extended clique-node matrix after Example \ref{malos1}, but that yields the graph $G'_Q=K_5$. Since $N[G]$ is an extended clique-node matrix, there must be $v\in V\setminus V(G')$ such that $\{v\}\cup V(G')$ is a  clique in $G_Q$, and moreover, its incidence vector is a row of $N[G]$.

\item Similarly, if $G'=C_6$, then  $N[G']=C_6^3$ which is not an extended clique-node matrix after Example \ref{malos1}, but that yields the graph $G'_Q=W^2_6$. In fact there are two maximal cliques in $W_6^2$, say $Q_1=\{1,3,5\}$ and $Q_2=\{2,4,6\}$, whose incidence vectors are not rows of $C_6^3$.  

Since $N[G]$ is an extended clique-node matrix, there must be $v,w\in V\setminus V(G')$ such that $S_1=\{v,1,2,3\}$ and  $S_2=\{w,2,4,6\}$ are cliques in $G_Q$, and their incidence vectors are rows of $N[G]$. If $v=w$, there is a node adjacent to every node in $G'$ and we are done.
If not, each of the nodes $v$ and $w$ generates three cycles of length 4 in $G'$. By the first case, each of them must have a universal node. This implies, recursively applying the previous case, that there must exist another node in $V$ adjacent to all the nodes in $G'$.  

\item Finally, if $G'=S_3$, the graph $G'_Q$ is a complete graph, and then there must be a node in $V\setminus V(G')$  adjacent to all the nodes in $G'$.
\end{itemize}

Now suppose $N[G]$ is not an extended clique-node matrix. According to Theorem \ref{cliquenodo}, $N[G]$
contains \[M=
\left(
\begin{array}{ccccccccc}
0&&1&&1&& 1& \dots &1\\
1&&0&&1&&1& \dots &1\\
1&&1&&0&&1& \dots &1
\end{array}
\right)
\]
as a submatrix, say, with columns $j_1,\dots,j_p$, for some $p\geq
3$, but it does not contain a row $i$ with $m_{ij_k}=1$ for $k=1,\dots,p$.

It is enough to prove that the submatrix $J-I$ of $M$ of order 3  
forces $G$ to have at least one induced subgraph  $G'$ of the family $\T$ and, by the assumption that $N[G]$ does not contain a row $i$ with $m_{ij_k}=1$ for $k=1,\dots,3$, we conclude that there is not a vertex $v\in V\setminus V(G')$ such that $N(v)\supset V(G')$. 

Clearly, the  submatrix $J-I$ of order 3 itself is not the neighbourhood matrix of any induced subgraph of $G$.  Then, $J-I$ is a submatrix of an induced subgraph of $G$, say $G^*$, with at least 4 nodes. Let us prove that  $G^*$ has at most 6 nodes.
Let us analyse the possible cases:
\begin{itemize}
\item If $|V(G^*)|=4$,  we have already seen that $N[G^*]=J-I$ is the  neighbourhood matrix of $C_4$.

\item If $|V(G^*)|=5$, say $V(G)=\{1,2,3,4,5\}$, w.l.o.g. some rows of $N[G^*]$ may be permuted in such a way that they correspond to the nodes listed on the table of the left-hand-side of Figure \ref{tablitas}. We have some freedom to choose the values of $\lambda_1, \lambda_2 \in \{0,1\}$. If $\lambda_1 =\lambda_2 =0$, then $N[G^*]$ is the neighbourhood matrix of $C_5$. If $\lambda_1\geq \lambda_2>0$ then $G^*$ has $C_4$ as a node induced subgraph (for instance $1-3-2-5-1$ when $\lambda_1=1$ and $\lambda_2=0$).

\item If $|V(G^*)|=6$, say $V(G)=\{1,2,3,4,5,6\}$, w.l.o.g. $N[G^*]$ may be indexed as in Figure \ref{tablitas}. Again, we have some freedom to choose the $0,1$ values of $\lambda_i$ for $i=1,\dots,6$.

If $\lambda_i=0$ for all $i$,  then $N[G^*]$ is the neighbourhood matrix of $C_6$.  If $\lambda_i=1$ for  $i\in \{1,2,3\}$ and $\lambda_i=0$ for  $i\in \{4,5,6\}$ (or $\lambda_i=0$ for  $i\in \{1,2,3\}$ and $\lambda_i=1$ for  $i\in \{4,5,6\}$) then $N[G^*]$ is the neighbourhood matrix of $S_3$ (Figure \ref{minimal}). If $\lambda_i=1$ for  all $i$,  then $N[G^*]$ is the neighbourhood matrix of the 3-pyramid also in Figure \ref{minimal}, which has  $C_4$  as a node induced subgraph.
It is easy to check that, for  the remaining possibilities for the $\lambda_i$'s and $i=1,\dots,6$, $G^*$ has $C_4$ or $C_5$ as a node induced subgraph (see Figure \ref{minimal}).
\end{itemize}

\begin{figure}[ht]
\centering
\begin{tabular}{c| c c c c c}
 & 1 & 2 & 3 & 4 & 5\\
\hline
2 & 0 & 1 & 1 & 0 & 1 \\
4 & 1 & 0 & 1 & 1 & $\lambda_2$ \\
5 & 1 & 1 & 0 & $\lambda_2$ & 1 \\
1 & 1 & 0 & $\lambda_1$ & 1 & 1 \\
3 & $\lambda_1$ & 1 & 1 & 1 & 0
\end{tabular}
\hspace{1.5cm} 
\begin{tabular}{c| c c c c c c c}
 & 1 & 2 & 3 & 4 & 5 & 6\\
\hline
4 & 0 & 1 & 1 & 1 & $\lambda_1$ & $\lambda_2$\\
5 & 1 & 0 & 1 & $\lambda_1$ & 1 & $\lambda_3$\\
6 & 1 & 1 & 0 & $\lambda_2$ & $\lambda_3$ & 1\\
1 & 1 & $\lambda_4$ & $\lambda_5$ & 0 & 1 & 1\\
2 & $\lambda_4$ & 1 & $\lambda_6$ & 1 & 0 & 1\\
3 & $\lambda_5$ & $\lambda_6$ & 1 & 1 & 1 & 0
\end{tabular}
\caption{$N[G^*]$ in the proof of Theorem \ref{main}.}\label{tablitas}
\end{figure}
\end{proof}

Recall that we need two conditions for a graph to belong to $\F$. We have explored the fact that $N[G]$ is an extended clique-node matrix. Let us now focus on the perfection of the clique graph associated with $G$.

\begin{remark}
We have already mentioned in Proposition \ref{webs} that the cycles $C_n$ for $n\geq 7$ do not belong to $\F$. In this case, this holds from the fact that the associated graph $G_Q=W^2_n$ is imperfect when $n>6$ due to Remark \ref{perfect_webs}. 
\end{remark}

\begin{proposition}\label{entusiasta}
If $G=(V,E)$ is such that
\begin{itemize}
\item there is $G'\subset G$ whose associated clique graph $G'_Q$ is imperfect (and thus $G'\notin \F$), 
\item for all $v\in V(G')$, $N_{G}(v)\subset N_{G}(w)$ for some $w \in V\setminus V(G')$  
\end{itemize}
then the clique graph $G_Q$ is imperfect, and then $G\notin \F$.
\end{proposition}

\begin{proof}
Let $U=V\setminus V(G')$ and $v\in U$. By hypothesis, $N_{G'}(v)\subset N_{G'}(w)$ for some $w \in V(G')$ and then the graph obtained from $G_Q$ after deleting the nodes in $U$ coincides with $G'_Q$. This shows that $G_Q$ is imperfect.  
\end{proof}

This result asks for some knowledge of subgraphs of $G$ not in $\F$ with  imperfect clique graph. We know that all cycles of length at least 7 satisfy this property. 
In fact, we can ask if the converse of this result is true, that is, is it enough to look for node induced cycles in the graph in order to have an  imperfect clique graph? 
The answer is negative as shown by the following infinite family of graphs.

\begin{example}\label{lastexample}
Let $G=(V,E)$ where $V=\{1,\dots,4k+2\}$ and $E$ is such that the even nodes in $V$ form a clique and $N(i)=\{i-1,i+1\}$ for every odd node $i\in V$.

In this case, the clique graph $G_Q$ is the union of an odd hole $C_{2k+1}$ and a complete graph $K_{2k+1}$, where also all the nodes in this complete graph are adjacent to every node in the cycle $C_{2k+1}$ (i.e., $G_Q$ is the complete join of the cycle and the complete graph). Clearly,  $G_Q$ is imperfect. 
Also, $G$ is a chordal graph, i.e., it has no cycle of length at least 4. 
\end{example}

\section{Conclusions}

In this paper we have analysed the relationship between the perfection of closed neighbourhood matrices and the $\{k\}$-Packing Function Problem in graphs.

We defined a family of graphs, $\F$, in which the $\{k\}$-Packing Function Problem can be solved in polynomial time. Then, this contribution  is  devoted to look  for a nice characterization of graphs in this family.

The definition of the family $\F$ and Chv\'atal's result ask for two conditions on the closed neighbourhood matrix of a graph in $\F$: the property of being a (extended) clique-node matrix and the perfection of its associated clique graph. 

Firstly, given a graph $G$, we have studied under which conditions the closed neighbourhood matrix $N[G]$ is an extended clique-node matrix and found a set of some small graphs which ---in case of being a subgraph of $G$ in the family--- need a universal node in $G$. 
We could completely characterized the  family of graphs that have extended clique-node matrices, although some of them have an imperfect clique graph.
   
Secondly, we provided some sufficient conditions for a graph not to be in the family  due to the imperfection of its associated  clique graph.

Based on the results in Proposition \ref{entusiasta} and Example \ref{lastexample}, it is clear that even if a graph satisfies the property of having an extended clique-node closed neighbourhood matrix,  it is important to study in more depth the imperfection of the associated clique graph  and, if possible, identify some minimal forbidden subgraphs for $G$.

In Figure \ref{diagrama} we have represented the relationship between the well-known graph classes of chordal and strongly chordal graphs and the family $\F$ introduced in this work. Also, we have depicted some other examples of graphs mentioned in this work (like web graphs, cycles, wheels and the graph $S_3$) in order to clarify our findings.

\begin{figure}[ht]
\centering
\includegraphics[scale=1.8]{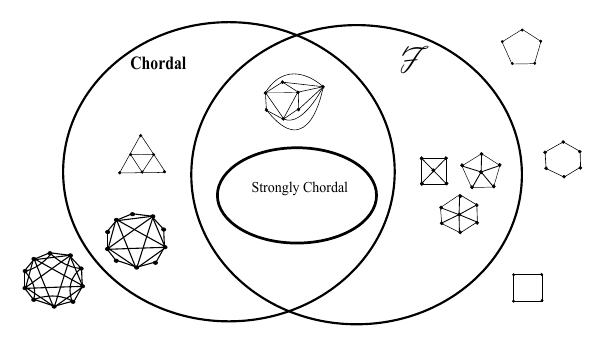}
\caption{}\label{diagrama}
\end{figure}

\end{document}